\documentclass[11pt]{amsart}
\usepackage{amsmath, amsthm, amsfonts, amssymb, amscd}
\usepackage{graphicx}
\usepackage{hyperref} 
\usepackage{algorithm}
\usepackage{algorithmic}
\usepackage{caption}
\usepackage{tikz}

\theoremstyle{plain}
\newtheorem{theorem}{Theorem}[section]
\newtheorem{lemma}[theorem]{Lemma}
\newtheorem{proposition}[theorem]{Proposition}

\theoremstyle{definition}
\newtheorem{definition}[theorem]{Definition}
\newtheorem{example}[theorem]{Example}

\theoremstyle{remark}

\providecommand{\set}[1]{\{#1\}}

\providecommand{\norm}[1]{\lVert#1\rVert}
\newcommand{\field}[1]{\mathbb{#1}}
\newcommand{\R}{\field{R}}

\newcommand{\Z}{\field{Z}}

\newcommand{\Vol}{\text{Vol}}
\renewcommand{\ker}{\operatorname{ker}}
\newcommand{\im}{\operatorname{im}}

\newcommand{\diam}{\operatorname{diam}}

\newcommand{\st}{\operatorname{st}}
\newcommand{\sd}{\operatorname{sd}}
\newcommand{\dist}{\operatorname{dist}}
\newcommand{\rad}{\operatorname{rad}}
\newcommand{\depth}{\operatorname{depth}}
\newcommand{\spn}{\operatorname{span}}

\title{A Cheeger-type inequality on simplicial complexes}
\author{John Steenbergen, Caroline Klivans, Sayan Mukherjee}
\address{Department of Mathematics, Duke University}
\address{Division of Applied Mathematics, Departments of Computer Science and Mathematics, Brown University}
\address{Departments of Statistical Science, Mathematics and Computer Science \\
Institute for Genome Sciences \& Policy, Duke University}
\date{\today}

\begin{document}

\begin{abstract}
In this paper, we consider a variation on Cheeger numbers related to the coboundary expanders recently defined by Dotterer and Kahle.  A Cheeger-type inequality is proved, which is similar to a result on graphs due to Fan Chung.  This inequality is then used to study the relationship between coboundary expanders on simplicial complexes and their corresponding eigenvalues, complementing and extending results found by Gundert and Wagner.  In particular, we find these coboundary expanders do not satisfy natural Buser or Cheeger inequalities.
\end{abstract}

\maketitle

\section{Introduction}
\subsection{Background}
The Cheeger inequality \cite{cheeger1970lower,buser36cheeger} is a classic result that relates the isoperimetric constant of a manifold
(with or without boundary) to the spectral gap of the Laplace-Beltrami operator. An analog of the manifold result was also found to hold
on graphs \cite{alon1985lambda,alon1986eigenvalues,mohar1989isoperimetric} and is a prominent result in spectral graph theory.
Given a graph $G$ with vertex set $V$, the Cheeger number is the following isoperimetric constant
\[ h := \min_{\emptyset \subsetneq S \subsetneq V} \frac{|\delta S|}{\min\set{|S|,|\overline{S}|}} \]
where $\delta S$ is the set of edges connecting a vertex in $S$ with a vertex in $\overline{S} = V \setminus S$. The Cheeger inequality on the graph relates the Cheeger number $h$ to the algebraic connectivity $\lambda$ \cite{fiedler1973algebraic} which is the the second eigenvalue of the graph Laplacian.  It states that
\[ 2h \geq \lambda \geq \frac{h^2}{2 \max_{v \in V} d_v} \]
where $d_v$ is the number of edges connected to vertex $v$ (also called the degree of the vertex). For more background on the Cheeger inequality see \cite{chung1997spectral}.

A key motivation for studying the Cheeger inequality has been understanding expander graphs \cite{hoory2006expander}  -- sparse graphs with strong connectivity properties. The edge expansion of a graph is the Cheeger number in these studies and expanders are families of regular graphs $\mathcal{G}$ of increasing size with the property $h(G) > \varepsilon$ for some fixed $\varepsilon > 0$ and all $G \in \mathcal{G}$.  A generalization of the Cheeger number to higher dimensions on simplicial complexes, based on ideas in \cite{linial2006homological,meshulam2009homological}, was
defined and expansion properties studied in \cite{dotterrer2010coboundary} via cochain complexes. In addition, it has long been known 
\cite{eckmann1944harmonische} that the graph Laplacian generalizes to higher dimensions on simplicial complexes.  In particular one can 
generalize the notion of algebraic connectivity to higher dimensions using the cochain complex and relate an eigenvalue of the $k$-dimensional Laplacian to the $k$-dimensional Cheeger number. This raises the question of whether the Cheeger inequality has a higher-dimensional analog.  

\subsection{Main Results}
In this paper we examine the combinatorial Laplacian which is derived from a chain complex and a cochain complex.  Precise definitions of the object studied
and the results are given in section 2. We first state our negative result -- for the cochain complex a natural Cheeger inequality does not hold. For an $m$-dimensional 
simplicial complex we denote $\lambda^{m-1}$ as the analog of the spectral gap for dimension $m-1$ on the cochain complex and we denote $h^{m-1}$ as the 
$(m-1)$-dimensional coboundary Cheeger number. In addition, let $S_k$ be the set of $k$-dimensional simplexes and for any $s \in S_k$ let $d_s$ be the number of 
$(k+1)$-simplexes incident to $s$. The following result is an informal statement of Proposition \ref{cochain} and implies that there exists no Cheeger inequality of the following form for
the cochain complex. Specifically, there are no constants $p_1, p_2, C$ such that either of the inequalities 
\[ C (h^{m-1})^{p_1} \geq \lambda^{m-1} \text{\qquad or \qquad} \lambda^{m-1} \geq \frac{C(h^{m-1})^{p_2}}{\max_{s \in S_{m-1}} d_s} \]
hold in general for an $m$-dimensional simplicial complex $X$ with $m>1$.  The case of $h^0$ and $\lambda^0$ with $p_1=1$ and $p_2 = 2$ reduces to the Cheeger inequality on the graph and the Cheeger inequality holds.

For the chain complex we obtain a positive result, there is a direct analogue for the Cheeger inequality in certain well-behaved cases.  Whereas the cochain complex is defined using the coboundary map, the chain
complex is defined using the boundary map.  Denote $\gamma_m$ as the analog of the spectral gap for dimension $m$ on the chain complex and $h_m$ as the $m$-dimensional Cheeger number defined using the boundary map.  If the $m$-dimensional simplicial complex $X$ is an orientable pseudomanifold or satisfies certain more general conditions, then
\[ h_m \geq \gamma_m \geq \frac{h_m^2}{2(m+1)}. \]
This inequality can be considered a discrete analog of the Cheeger inequality for manifolds with Dirichlet boundary condition \cite{cheeger1970lower,buser36cheeger}.

\subsection{Related Work}
A probabilistic argument was used by Gundert and Wagner \cite{gundert2012laplacians} to show on the cochain complex there exists infinitely many simplicial complexes
with  $h^{m-1} = 0$ and $\lambda^{m-1} > c$ for some fixed constant $c > 0$ -- implying that one side of the Cheeger inequality cannot hold in general. However,
this construction requires the complexes to have torsion in their integral homology groups due to the way $h^{m-1}$ and $\lambda^{m-1}$ relate to cohomology. In this paper
we show that even for torsion-free simplicial complexes there exist counterexamples that rule out both sides of a Cheeger inequality.

The analysis of the chain complex in our paper is related to a paper by Fan Chung \cite{chung2007random} which introduces a notion of a Cheeger number
on graphs with the analog of a Dirichlet boundary condition. We provide a detailed comparison on Appendix A.

Finally, it should be mentioned that the authors in \cite{parzanchevski2012isoperimetric} prove a two-sided Cheeger-type inequality for $\lambda^{m-1}$ using a modified higher-dimensional Cheeger number.  The modified Cheeger number used is nonzero only if the simplicial complex has complete skeleton, and the Cheeger side of the inequality includes an additive constant.

\section{Main Results}
\subsection{Simplicial Complexes}
Since the concept of a  Cheeger inequality is strongly associated to manifolds we focus in this paper on abstract simplicial complexes that are analogous to well-behaved manifolds. In particular, we will focus on simplicial complexes that have geometric realizations homeomorphic to a Euclidean ball $B^m:=\set{x \in \R^m : \norm{x}_2 \leq 1}$.  We will call such complexes simplicial $m$-balls

By a simplicial complex we always mean an abstract finite simplicial complex.  Simplicial complexes generalize the notion of a graph to higher dimensions.  Given a set of vertices $V$, any nonempty subset $\sigma \subseteq V$ of the form $\sigma = \set{v_0, v_1, \ldots, v_k}$ is called a $k$-dimensional simplex, or $k$-simplex.  A simplicial complex $X$ is a finite collection of simplexes of various dimensions such that $X$ is closed under inclusion, i.e., $\tau \subseteq \sigma$ and $\sigma \in X$ implies $\tau \in X$.

Given a simplicial complex $X$ denote the set of $k$-simplexes of $X$ as $S_k := S_k(X)$.  We call $X$ a simplicial $m$-complex if $S_m(X) \neq \emptyset$ but $S_{m+1}(X) = \emptyset$.  Given two simplexes $\sigma \in S_k$ and $\tau \in S_{k+1}$ such that $\sigma \subset \tau$, we call $\sigma$ a face of $\tau$ and $\tau$ a coface of $\sigma$.  Two $k$-simplexes are lower adjacent if they share a common face and are upper adjacent if they share a common coface.  

Every simplicial complex $X$ has associated with it a geometric realization denoted $|X|$.  The simplicial $m$-complex $\Sigma^m$ consisting of a single $m$-simplex and its subsets has geometric realization homeomorphic to $B^m$.  Thus, $\Sigma^m$ is an example of a simplicial $m$-ball.  A subdivision of a simplicial complex $X$ is a simplicial complex $X'$ such that $|X'| = |X|$ and every simplex of $X'$ is, in the geometric realization, contained in a simplex of $X$.  Thus, any subdivision of $\Sigma^m$ is also a simplicial $m$-ball.

There is another convenient set of criteria under which a simplicial complex is a simplicial $m$-ball.  A simplicial $m$-complex $X$ is constructible if either (1) $X = \Sigma^m$ or (2) $X$ can be decomposed into the union of two constructible simplicial $m$-subcomplexes $X = X_1 \cup X_2$ such that $X_1 \cap X_2$ is a constructible simplicial $(m-1)$-complex.  If every $s \in S_{m-1}$ has at most two cofaces then $X$ is said to be non-branching.  In this case, every $s \in S_{m-1}$ with exactly one coface is called a boundary face of $X$.  It is known \cite{bjorner1995topological} that a the geometric realization of a non-branching constructible simplicial $m$-complex $X$ is homeomorphic to $B^m$ if $X$ has at least one boundary face (otherwise it is homeomorphic to the sphere).

\subsection{Chain and Cochain Complexes}
Given a simplicial complex $X$ and any field $F$, we can define the chain and cochain complexes of $X$ over $F$. 
In this paper we consider the fields $\Z_2$ and $\R$.  Given a simplex $\sigma = \set{v_0, v_1, \ldots, v_k}$, $\sigma$ can be ordered as a set.  An orientation, denoted by $[v_0, v_1, \ldots, v_k]$ is an equivalence class of all even permutations of the given ordering.  There are always two orientations for $k > 0$.  The space of $k$-chains $C_k(F):=C_k(X;F)$ is the vector space of linear combinations of oriented $k$-simplexes with coefficients in $F$, with the stipulation that the two orientations of a simplex are negatives of each other in $C_k(F)$.  The space of $k$-cochains $C^k(F):=C^k(X;F)$ is then defined to be the vector space dual to $C_k(F)$.  These spaces are isomorphic and we will make no distinction between them.  The boundary map $\partial_k(F):C_k(F) \to C_{k-1}(F)$ is defined on the basis elements $[v_0, \ldots, v_k]$ as
\[ \partial_k[v_0, \ldots, v_k] = \sum_{i=0}^k (-1)^i[v_0, \ldots, v_{i-1}, v_{i+1}, \ldots, v_k] \]
The coboundary map $\delta^{k-1}(F):C^{k-1}(F) \to C^k(F)$ is then defined to be the transpose of the boundary map.  When there is no confusion, we will denote the boundary and coboundary maps by $\partial$ and $\delta$.  It is easy to see that $\partial \partial = \delta \delta = 0$, so that $(C_k(F), \partial_k)$ and $(C^k(F), \delta^k)$ form chain and cochain complexes.  See Figures \ref{boundarymap} and \ref{coboundarymap} for examples of $\partial$ and $\delta$ on real and $\Z_2$ chains/cochains.

\begin{figure}
\centering
\resizebox{\textwidth}{!}{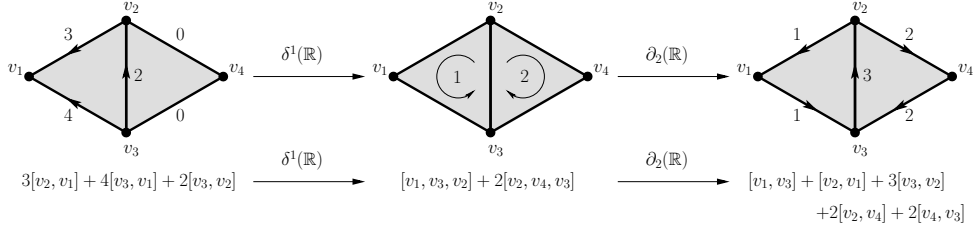}\\
\caption{An example of $\partial(\R)$ and $\delta(\R)$.}
\label{boundarymap}
\end{figure}
\begin{figure}
\centering
\resizebox{\textwidth}{!}{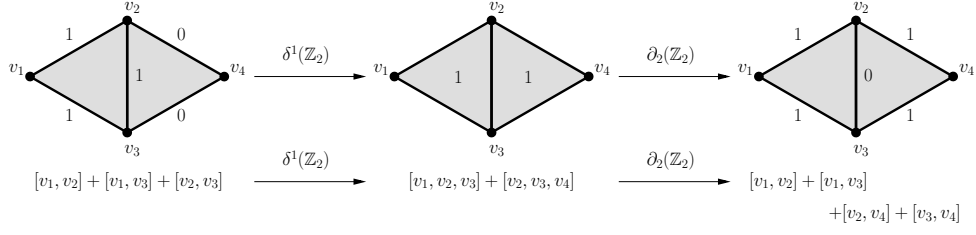}\\
\caption{An example of $\partial(\Z_2)$ and $\delta(\Z_2)$.}
\label{coboundarymap}
\end{figure}

When $F = \Z_2$, positive and negative have no meaning and therefore no distinction is made between different orientations.  In particular, it is possible to identify $C_k(\Z_2)$ and $C_k(\Z_2)$ with $S_k$ as sets.  Throughout this paper, we will identify a $k$-chain/$k$-cochain $\phi$ over $\Z_2$ with the subset $\phi \subset S_k$ of $k$-simplexes to which $\phi$ assigns the coefficient 1.

The homology and cohomology vector spaces of $X$ over $F$ are 
\[ H_k(F):= H_k(X;F) = \frac{\ker \partial_k}{\im \partial_{k+1}} \text{\quad and \quad} H^k(F):= H^k(X;F) = \frac{\ker \delta^k}{\im \delta^{k-1}}. \]
It is known from the universal coefficient theorem that $H^k(F)$ is the vector space dual to $H_k(F)$.

\subsection{Laplacians and Eigenvalues}
The $k$-th Laplacian of $X$ is defined to be 
\[ L_k := L_k^\text{up} + L_k^\text{down} \]
where 
\[L_k^\text{up} = \partial_{k+1}(\R)\delta^k(\R) \text{\quad and \quad} L_k^\text{down} = \delta^{k-1}(\R)\partial_k(\R).\]
By way of Rayleigh quotients, the smallest nontrivial eigenvalue of $L_k^\text{up}$ and $L_k^\text{down}$ are given by
\[ \lambda^k  = \min_{\substack{f \in C^k(\R) \\ f \perp \im \delta}} \frac{\norm{\delta f}_2^2}{\norm{f}_2^2} = \min_{\substack{f \in C^k(\R) \\ f \notin \im \delta}} \frac{\norm{\delta f}_2^2}{\min_{g \in \im \delta} \norm{f+g}_2^2}, \]
\[ \lambda_k =  \min_{\substack{f \in C_k(\R) \\ f \perp \im \partial}} \frac{\norm{\partial f}_2^2}{\norm{f}_2^2} = \min_{\substack{f \in C_k(\R) \\ f \notin \im \partial}} \frac{\norm{\partial f}_2^2}{\min_{g \in \im \partial} \norm{f+g}_2^2}, \]
where $\norm{\cdot}_2$ denotes the Euclidean norm on both $C^k(\R)$ and $C_k(\R)$.  It is well known that the nonzero spectrum of $L_k$ is the union of the nonzero spectrum of $L_k^\text{up}$ with the nonzero spectrum of $L_k^\text{down}$.  Thus, the smallest nonzero eigenvalue of $L_k$ is either $\lambda^k$ or $\lambda_k$ assuming one of them is nonzero.  In addition, the nonzero spectrum of $L_k^\text{up}$ is the same as the nonzero spectrum of $L_{k+1}^\text{down}$.  Thus, $\lambda^k = \lambda_{k+1}$ whenever $\lambda^k, \lambda_{k+1}$ are both nonzero.

The relationship between eigenvalues and homology/cohomology is as follows:
\begin{center}
\begin{tabular}{c c c} 
$\lambda_k = 0$ & & $\lambda^k = 0$ \\ 
$\Updownarrow$ &\text{\quad and \quad} &$\Updownarrow$ \\ 
$H_k(\R) \neq 0$ & & $H^k(\R) \neq 0$. 
\end{tabular}
\end{center}
If we pass to the reduced cochain complex, $\lambda^0$ becomes the algebraic connectivity (or Fiedler number) of a graph \cite{fiedler1973algebraic} and $\lambda^0 = 0 \Leftrightarrow \widetilde{H}^0(\R) \neq 0$.

\subsection{Cheeger Numbers}
Higher-dimensional Cheeger numbers were first stated in \cite{dotterrer2010coboundary} to capture a higher-dimensional notion of
expanders. They are defined via the coboundary map as follows:
\begin{definition}
 Let $\norm{\cdot}$ denote the Hamming norm on $C^k(\Z_2)$.  The $k$-th (coboundary) Cheeger number of $X$ is 
\[ h^k := \min_{\substack{\phi \in C^k(\Z_2) \\ \phi \notin \im \, \delta}} \frac{\norm{\delta\phi}}{\min_{\psi \in \im \, \delta}\norm{\phi+\psi}}. \]
\end{definition}

A similar definition can be given for the boundary map.
\begin{definition}
Let $\norm{\cdot}$ also denote the Hamming norm on $C_k(\Z_2)$.
The $k$-th boundary Cheeger number of $X$ is 
\[ h_k := \min_{\substack{\phi \in C_k(\Z_2) \\ \phi \notin \im \, \partial}} \frac{\norm{\partial\phi}}{\min_{\psi \in \im \, \partial}\norm{\phi+\psi}}. \]
%In particular, the top-dimensional Dirichlet Cheeger number of $X$ is.
%\[h_m = \min_{\substack{\phi \in C_k(\Z_2) \\ \phi \neq 0}} \frac{\norm{\partial\phi}}{\norm{\phi}} \]
\end{definition}

The relationship between Cheeger numbers and homology/cohomology is as follows:
\begin{center}
\begin{tabular}{c c c} 
$h_k = 0$ & & $h^k = 0$ \\ 
$\Updownarrow$ &\text{\quad and \quad} &$\Updownarrow$ \\ 
$H_k(\Z_2) \neq 0$ & & $H^k(\Z_2) \neq 0$ .
\end{tabular}
\end{center}
If we pass to the reduced cochain complex, $h^0$ becomes the Cheeger number of a graph \cite{dotterrer2010coboundary}
and $h^0 = 0 \Leftrightarrow \widetilde{H}^0(\Z_2) \neq 0$.

Often, we speak of a cochain that attains the minimum in the definition of the Cheeger number  -- in the graph case these are  Cheeger cuts.  We will say that $\phi \in C^k(\Z_2)$ attains $h^k$ if $h^k = \frac{\norm{\delta \phi}}{\norm{\phi}}$.  The same terminology will be used for $h_k$.

\subsection{Additional Notation and Preliminary Results}

Here we collect some interesting results concerning Cheeger numbers which will be needed later in section \ref{main results}.  Lemma \ref{diam} says that $h_1$ has a very simple interpretation in terms of the diameter of the simplicial complex.  Lemma \ref{lem: X_k} says that $h^{m-1}$ also has a very simple interpretation in terms of the radius.  

We define the diameter of a simplicial $m$-complex $X$ as follows.  Given two vertices $v_1, v_2 \in S_0$, we define the distance between them to be the quantity
\[ \dist(v_1, v_2) := \min\set{\norm{\phi} : \phi \in C_1(\Z_2) \text{ and } \partial \phi = v_1 + v_2} \]
Any chain $\phi$ attaining the minimum is called a geodesic.  Note that for any geodesic $\phi$, $h_1 \leq \frac{2}{\norm{\phi}}$.  For our purposes, $\dist(v_1, v_2) = 0$ if $v_1$, $v_2$ are not in the same connected component.  The diameter of $X$ is then defined to be 
\[ \diam(X) := \max_{v_1, v_2 \in S_0} \dist(v_1, v_2). \]
As it turns out, $h_1$ is strongly related to the diameter of a simplicial complex.
\begin{lemma}\label{diam}
Given a simplicial $m$-complex $X$ with $m \geq 1$ and satisfying $H_1(\Z_2) = 0$, $h_1$ is attained by a geodesic and hence
\[ h_1 = \frac{2}{\diam(X)} \].
\end{lemma}
\begin{proof}
Suppose that $\phi \in C_1(\Z_2)$ attains $h_1$.  Clearly, $\norm{\partial \phi}$ must be even and nonzero.  What we will show is that we can assume $\norm{\partial \phi} = 2$.  Thinking of $\phi$ as a graph (consisting of the edges in $\phi$ and their vertices), it is also clear that every connected component $\phi_i$ of $\phi$ has $\norm{\partial \phi_i}$ even.  For every pair of vertices in $\partial \phi_i$, there exists a geodesic in $X$ with the given pair of vertices as its boundary.  Thus, there exist geodesics $\psi_1, \ldots, \psi_q$ such that $\partial \psi_j$ is a distinct pair of vertices in $\partial \phi$ for all $j$ and $\partial (\psi_1 + \cdots + \psi_q) = \partial \phi$.  Since $\phi$ attains $h_1$ and $H_1(\Z_2) = 0$, 
\[ \norm{\phi} = \min_{\psi \in \im \partial}\norm{\phi + \psi} = \min_{\partial \psi = \partial \phi} \norm{\psi} \] 
In other words, $\phi$ is a 1-chain of smallest norm with boundary $\partial \phi$.  Thus, $\norm{\psi_1 + \cdots + \psi_q} \geq \norm{\phi}$.  Now,
\begin{align*}
h_1 &= \frac{\norm{\partial \phi}}{\norm{\phi}} \\
    &\geq \frac{\norm{\partial (\psi_1 + \cdots + \psi_q)}}{\norm{\psi_1 + \cdots + \psi_q}} \\
    &\geq \frac{2 + \cdots + 2}{\norm{\psi_1} + \cdots + \norm{\psi_q}} \\
    &\geq \min\left\{ \frac{2}{\norm{\psi_1}}, \ldots, \frac{2}{\norm{\psi_q}} \right\} \\
    &\geq h_1
\end{align*}
and therefore $h_1 = \min\big\{ \frac{2}{\norm{\psi_1}}, \ldots, \frac{2}{\norm{\psi_q}} \big\}$.  Here we are using the general inequality $\frac{a_1 + a_2 + \cdots + a_k}{b_1 + b_2 + \cdots + b_k} \geq \min_i \frac{a_i}{b_i}$, valid for all $a_1, \ldots, a_k, b_1, \ldots, b_k > 0$.  Hence, $h_1 = \frac{2}{\norm{\psi_j}}$ for some geodesic $\psi_j$.  This completes the proof.
\end{proof}

While the diameter is defined in terms of 1-chains, we define the radius in terms of $(m-1)$-cochains as follows.  Given a simplicial $m$-complex $X$, we define the depth of an $m$-simplex $\sigma$ to be
\[ \depth(\sigma) := \min \set{\norm{\phi} : \phi \in C^{m-1}(\Z_2), \delta \phi = \sigma}. \]
Any minimizing $\phi$ will be said to be a depth-attaining cochain for $\sigma$.  Note that for any such $\phi$, $h^{m-1} \leq \frac{1}{\norm{\phi}}$.  All $m$-simplexes have a defined depth when $H_m(\Z_2)$ is trivial.  In this case, we define the radius of $X$ to be 
\[ \rad(X) := \max_{\sigma \in S_m} \depth(\sigma). \]

Depth-attaining cochains have a very predictable structure for non-branching simplicial complexes, a fact which we will use later in proving Proposition \ref{cochain}.  Roughly speaking, Lemma \ref{lem: depth-attaining cochains} says that if $\phi$ is depth-attaining for $\sigma$, then $\phi$ is a linear non-intersecting sequence of $(m-1)$-simplexes starting with a face of $\sigma$ and ending with a boundary face.  For the statement and proof of this Lemma we define the star $\st(s)$ of a simplex $s$ to be the set of cofaces of $s$. 

\begin{lemma} \label{lem: depth-attaining cochains}
Let $X$ be a simplicial $m$-complex such that every $s \in S_{m-1}$ has at most two cofaces.  Suppose that $\sigma \in S_m$ has depth $d$ and $\phi$ is a depth-attaining cochain for $\sigma$.  Then there is a sequence $s_1, s_2, \ldots, s_d$ of distinct $(m-1)$-simplexes and a sequence $\sigma = \sigma_1, \sigma_2, \ldots, \sigma_d$ of distinct $m$-simplexes satisfying
\begin{enumerate}
 \item $\phi = \sum_{i=1}^d s_i$,
 \item $\st(s_i) = \set{\sigma_i, \sigma_{i+1}}$ for $i < d$,
 \item $\st(s_d) = \set{\sigma_d}$.
\end{enumerate}
\end{lemma}
%Before beginning the proof, we observe that conditions 2 and 3 are equivalent to saying the following: $s_1$ is a face of $\sigma$, $s_d$ is a boundary face, $s_i$ is upper adjacent to $s_{i+1}$ for $i < d$, and no $m$-simplex has more than two of the $s_i$ as faces.
\begin{proof}
Assume $\phi = \sum_{i=1}^d s_i$. Clearly, at least one of the $s_i$ must have $\sigma$ as a coface, so WLOG we can assume $s_1$ has $\sigma = \sigma_1$ as a coface.  If $s_1$ is a boundary face, we are done and $d=1$.  If not, then $s_1$ has another coface $\sigma_2$.  In this case, if there are no other $s_i$ with $\sigma_2$ as a coface then we arrive at the contradiction that $\delta \phi$ contains $\sigma_2$, i.e., $\delta \phi \neq \sigma$.  Thus, there is another $s_i$ with $\sigma_2$ as a coface, which we can assume WLOG is $s_2$.  

We proceed by induction.  Suppose that for $k > 1$ there is a sequence $\sigma_1, \sigma_2, \ldots, \sigma_{k}$ of distinct $m$-simplexes such that $\delta(s_1 + \cdots + s_{k-1}) = \sigma + \sigma_k$ where $\st(s_i) = \set{\sigma_i, \sigma_{i+1}}$ for all $i$.  Then we can find another $s_i$, $i > k$, which we can assume WLOG is $s_k$ and which has $\sigma_k$ as a coface.  If no such $s_i$ exists then $\delta \phi \neq \sigma$.  If $s_k$ is a boundary face we are done and $d = k$.  If $s_k$ has $\sigma_{k+1}$ as a second coface and $\sigma_{k+1} = \sigma_i$ for some $i < k$ then $s_i + \ldots + s_k$ is a cocycle, but this means that $\delta(\phi - s_i - \cdots - s_k) = \sigma$ so  $\phi$ is not depth-attaining.  Otherwise, $\sigma_1, \sigma_2, \ldots, \sigma_{k+1}$ is a sequence of distinct $m$-simplexes such that  $\delta(s_1 + \cdots + s_k) = \sigma + \sigma_{k+1}$ where $\st(s_i) = \set{\sigma_i, \sigma_{i+1}}$ for all $i$.  This leaves us back where we started.  By induction, we can continue this process until $k=d$ and $s_d$ is a boundary face.  
\end{proof}

\begin{lemma} \label{lem: X_k}
Let $X$ be a simplicial $m$-complex with $H^{m-1}(\Z_2) = 0$ and $H_m(\Z_2) = 0$.  Then $h^{m-1}$ is attained by a depth-attaining cochain and hence
\[ h^{m-1} = \frac{1}{\rad(X)}. \]
\end{lemma}
\begin{proof}
Suppose $\psi$ attains $h^{m-1}$ and $\delta \psi$ is a sum of distinct $m$-simplexes $\sigma_1, \ldots, \sigma_q$ with depth-attaining cochains $\psi_1, \ldots, \psi_q$.  Clearly $\norm{\psi} \leq \norm{\psi_1} + \cdots + \norm{\psi_q}$, so 
\begin{align*}
h^{m-1} &= \frac{q}{\norm{\psi}} \\
        &\geq \frac{1+\cdots+1}{\norm{\psi_1} + \cdots + \norm{\psi_q}} \\
        &\geq \min\left\{ \frac{1}{\norm{\psi_1}}, \ldots, \frac{1}{\norm{\psi_q}} \right\} \\
        &\geq h^{m-1}
\end{align*}
and therefore $h^{m-1} = \min \big\{ \frac{1}{\norm{\psi_1}}, \ldots, \frac{1}{\norm{\psi_q}} \big\}$.  Here we are using the general inequality $\frac{a_1 + a_2 + \cdots + a_k}{b_1 + b_2 + \cdots + b_k} \geq \min_i \frac{a_i}{b_i}$, valid for all $a_1, \ldots, a_k, b_1, \ldots, b_k > 0$.  Hence, $h^{m-1} = \frac{1}{\norm{\psi_j}}$ for some depth-attainng cochain $\psi_j$.  This completes the proof.
\end{proof}

An interesting result which will not be used in this paper is a Cheeger-type inequality for the special case $X = \Sigma^m$.
\begin{lemma} Recall $\Sigma^m$ is the simplicial complex induced by an $m$-simplex.  The following holds for all $k$.
 \begin{enumerate}
  \item $h^k(\Sigma^{m-1}) \geq \frac{m}{k+2}$
  \item $h_k(\Sigma^{m-1}) \geq \frac{m}{m-k}$.
 \end{enumerate}
\end{lemma}
The reason this result is Cheeger-type is because all the Laplacian eigenvalues of all dimensions for $\Sigma^{m-1}$ are equal to $m$ (this is easily seen from the characterization of the Laplacian in \cite{muhammad2006control}).  Part (1) of this Lemma was proved by Meshulam and Wallach \cite{meshulam2009homological} (who, even though they did not define the Cheeger number, still worked with its numerator and denominator separately).  Their proof can be easily modified to prove part (2) of the Lemma.

\subsection{Main Results}\label{main results}

We now state the main results of this paper -- there exists a Cheeger-type inequality in the top dimension for the chain complex but not for the cochain complex.

To state the results we need the following notion of orientational similarity.  Two oriented lower adjacent $k$-simplexes are dissimilarly oriented if they induce the same orientation on the common face.  In other words, if $\sigma = [v_0, \ldots, v_k]$ and $\tau = [w_0, \ldots, w_k]$ share the face $\set{u_0, \ldots, u_{k-1}}$, then $\sigma$ and $\tau$ are dissimilarly oriented if $\partial(\R) \sigma$ and $\partial(\R) \tau$ assign the same coefficient ($+1$ or $-1$) to the oriented simplex $[u_0, \ldots, u_{k-1}]$.  Otherwise, they are said to be similarly oriented.  If $X$ is a simplicial $m$-complex and all its $m$-simplices can be oriented similarly, then $X$ is called orientable.

 We first state the positive result -- there is a Cheeger-type inequality for the chain complex. 
 \begin{theorem}\label{Buser_Cheeger}
  Let $X$ be a simplicial $m$-complex, $m > 0$. 
  \begin{enumerate}
  \item[(1)] Let $\phi \in C_m(\Z_2)$ minimize the quotient in 
  \[ h_m := \min_{\substack{\phi \in C_m(\Z_2) \\ \phi \notin \im \, \partial}} \frac{\norm{\partial\phi}}{\min_{\psi \in \im \, \partial}\norm{\phi+\psi}} .\]
 If all $m$-simplexes in $\phi$ can be similarly oriented, then $h_m \geq \lambda_m$.
 \item[(2)] Assume that every $(m-1)$-dimensional simplex is incident to at most two $m$-simplexes.  Then
\[ \lambda_m \geq \frac{h_m^2}{2(m+1)}. \]
  \end{enumerate}
 
 \end{theorem}
 
The first statement is the analog of the Buser inequality for graphs.
The second statement is an analog of the Cheeger inequality for graphs, as well as the Cheeger inequality for a manifold with Dirichlet 
boundary conditions. The constraint that every $(m-1)$-simplex has at most two cofaces enforces the
boundary condition. The hypotheses required for both inequalities are always satisfied by orientable pseudomanifolds. 

The hypotheses required by the Theorem cannot be removed, as proved by the following two examples.
\begin{example}[Real Projective Plane] \label{RP2Example}
Given a triangulation $X$ of $\R P^2$ (see Figure \ref{fund_poly}) we know that $H_2(\Z_2) \neq 0$ while $H_2(\R) = 0$, so that $h_2 = 0 \neq \lambda_2$.  This is due to the nonorientability of $\R P^2$.  The chain $\phi \in C_2(\Z_2)$ containing every $m$-simplex has no boundary.  However, the $m$-simplexes cannot all be similarly oriented, so that there is no corresponding boundaryless chain in $C_2(\R)$.  As a result, the hypothesis used in part (1) of the Theorem cannot in general be removed.
\begin{figure}[h]
\centering
\resizebox{0.3\textwidth}{!}{\includegraphics{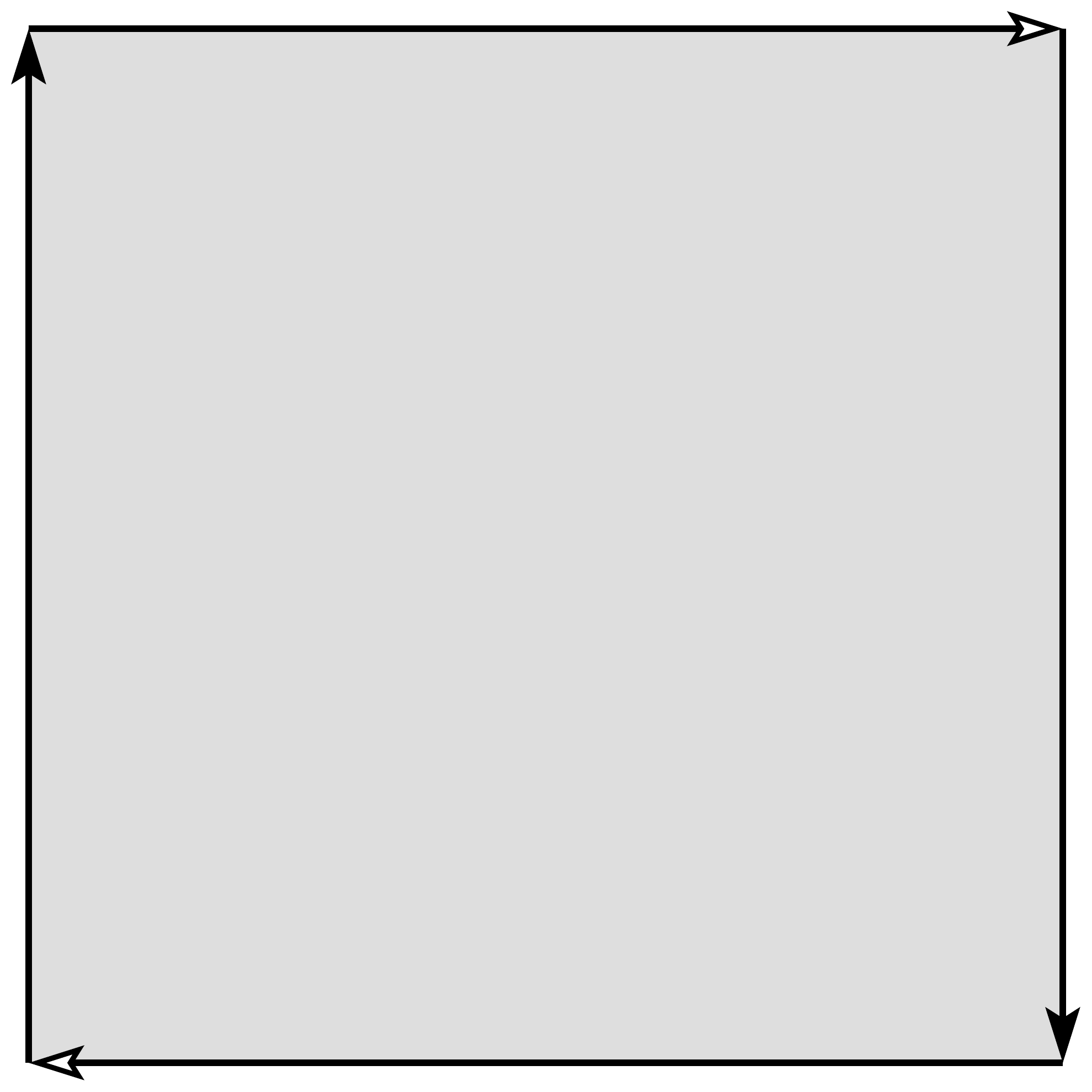}}\\
\caption{The fundamental polygon of $\R P^2$.}
\label{fund_poly}
\end{figure}
\end{example}

\begin{example}
Let $G_k$ be a graph with $2k$ vertices of degree one, half of which connect to one end of an edge and the other half connect to the other end (see figure \ref{graph}).  Clearly, $h^0(G_k) = \frac{1}{k+1}$ while Lemma \ref{diam} implies $h_1 = \frac{2}{3}$.  By the Buser inequality for graphs, $\lambda^0 \leq \frac{2}{k+1}$ and since $\lambda_1 = \lambda^0$, this means that $\lambda_1 \to 0$.  As a result, we conclude that the hypothesis used in part (2) of the Theorem cannot be removed.
\begin{figure}[h]
\centering
\resizebox{0.5\textwidth}{!}{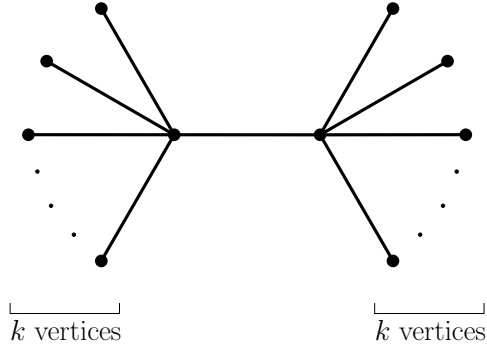}\\
\caption{The family of graphs $G_k$.}
\label{graph}
\end{figure}
\end{example}

\begin{proof}[Proof of Theorem \ref{Buser_Cheeger}] 
Given the hypotheses, $\lambda_m$ is a linear programming relaxation of $h_m$.  Let $g \in C_m(\R)$ be the chain which assigns a 1 to every simplex in $\phi$ (all of them similarly oriented) and a 0 to every other simplex.  Then
\begin{displaymath}
h_m = \frac{\norm{\partial \phi}}{\norm{\phi}} = \frac{\norm{\partial g}_2^2}{\norm{g}_2^2} \geq \min_{\substack{f \in C_m(\R) \\ f \neq 0}} \frac{\norm{\partial f}_2^2}{\norm{f}_2^2} = \lambda_m. 
\end{displaymath}
\end{proof}
\begin{proof}[Proof of Theorem \ref{Buser_Cheeger}]
Let $f$ be an eigenvector of $\lambda_m$ and for any oriented $m$-simplex $\sigma$ let $f(\sigma)$ denote the coefficient assigned to $\sigma$ by $f$.  Orient the $m$-simplexes of $X$ so that all the values of $f$ are non-negative and let $S_m^\text{or}(X)$ be the set of oriented $m$-simplices of $X$.  We do not assume the $m$-simplexes are similarly oriented.  Number the $m$-simplexes from $1$ to $N := |S_m^\text{or}(X)|$ in increasing order of $f$:
\[ 0 \leq f(\sigma_1) \leq f(\sigma_2) \leq \cdots \leq f(\sigma_N). \]

To aid us in the proof, we introduce a new simplicial $m$-complex $X'$ which contains $X$ as a subcomplex and which is defined as follows: for every boundary face $s = \set{v_0, \ldots, v_{m-1}}$ in $X$ create a new vertex $v$ and a new $m$-simplex $\sigma = \set{v_0, \ldots, v_{m-1},v}$ which includes $v$ and $s$.  These new $m$-simplexes will be called border facets.  Give the border facets any orientation and let $F_m^\text{or}(X')$ be the set of oriented border facets.  We can extend $f$ to be a function on $S_m^\text{or}(X) \cup F_m^\text{or}(X')$ by defining $f(\sigma)=0$ for any $\sigma \in F_m^\text{or}(X')$.  Let $M := |F_m^\text{or}(X')|$ and number the oriented border facets in any order: 
\[ F_m^\text{or}(X') = \set{\sigma_0, \sigma_{-1}, \ldots, \sigma_{1-M}}.\]

\begin{figure}
\centering
\resizebox{0.6\textwidth}{!}{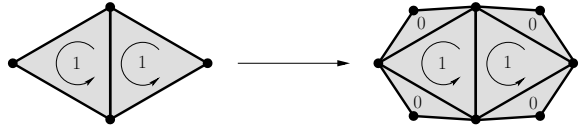}\\
\caption{Making Dirichlet boundary conditions explicit.}
\label{phantom}
\end{figure}

The intuition behind introducing the border facets comes from the analogy with the continuous Cheeger inequality for functions satisfying Dirichlet boundary conditions (see \cite{cheeger1970lower}).  In our case, the Dirichlet boundary condition is implicit in the fact that $f$ is defined on $m$-simplexes (as opposed to vertices).  The border facets represent the boundary of the $m$-dimensional part of $X$, and $f$ is in fact zero on them.  See Figure \ref{phantom} for a depiction.  In this analogy, $h_m$ plays the part of the Cheeger number defined as in \cite{cheeger1970lower} for manifolds with boundary.

When two simplexes $\sigma, \tau$ are lower adjacent we write $\sigma \sim \tau$.  Now define 
\begin{displaymath}
C_i = \left\{ \set{\sigma_j, \sigma_k} : 1-M \leq j \leq i < k \leq N \text{ \ and \ } \sigma_j \sim \sigma_k \right\}
\end{displaymath}
and
\begin{displaymath}
 h[f] = \min_{0 \leq i \leq N-1} \frac{ |C_i|}{N-i}. 
\end{displaymath}
Observe that $h[f] \geq h_m$.

We now finish the theorem.  The following summations are taken over all oriented $m$-simplexes in $S_m^\text{or}(X) \cup F_m^\text{or}(X')$.
%{\allowdisplaybreaks
\begin{align}
 \lambda_m &= \frac{\sum_{\sigma \sim \tau} (f(\sigma) \pm f(\tau))^2}{\sum_{\sigma} f(\sigma)^2}, \tag{1} \\
           &= \frac{\sum_{\sigma \sim \tau} (f(\sigma) \pm f(\tau))^2}{\sum_{\sigma} f(\sigma)^2} \cdot \frac{\sum_{\sigma \sim \tau} (f(\sigma) \mp f(\tau))^2}{\sum_{\sigma \sim \tau} (f(\sigma) \mp f(\tau))^2}, \nonumber \\ 
           &\geq \frac{\left( \sum_{\sigma \sim \tau} |f(\sigma)^2 - f(\tau)^2| \right)^2}{\left( \sum_{\sigma} f(\sigma)^2 \right) \cdot \left(\sum_{\sigma \sim \tau} (f(\sigma) \mp f(t_2))^2\right)}, \tag{3} \\ 
           &\geq \frac{\left( \sum_{\sigma \sim \tau} |f(\sigma)^2 - f(\tau)^2| \right)^2}{\left( \sum_{\sigma} f(\sigma)^2 \right) \cdot \left(2\sum_{\sigma \sim \tau}f(\sigma)^2 + f(\tau)^2\right)}, \nonumber \\ 
           &= \frac{\left( \sum_{\sigma \sim \tau} |f(\sigma)^2 - f(\tau)^2| \right)^2}{\left( \sum_{\sigma} f(\sigma)^2 \right) \cdot 2(m+1) \cdot \left(\sum_{\sigma} f(\sigma)^2\right)}, \nonumber \\ 
           &= \frac{\left( \sum_{i=0}^{N-1}(f(\sigma_{i+1})^2 - f(\sigma_i)^2)|C_i|\right)^2}{2(m+1) \cdot \left(\sum_{\sigma} f(\sigma)^2\right)^2}, \tag{6} \\ 
           &\geq \frac{\left( \sum_{i=0}^{N-1}(f(\sigma_{i+1})^2 - f(\sigma_i)^2)h[f](N-i)\right)^2}{2(m+1) \cdot \left(\sum_{\sigma} f(\sigma)^2\right)^2}, \nonumber \\
           &= \frac{h[f]^2}{2(m+1)} \cdot \frac{\left(\sum_{\sigma} f(\sigma)^2\right)^2}{\left(\sum_{\sigma} f(\sigma)^2\right)^2}, \nonumber \\
           &\geq \frac{h_m^2}{2(m+1)}. \nonumber
\end{align}
%}
Step (1) follows from the Rayleigh quotient characterization of $\lambda_m$ and step (3) follows from the Cauchy-Schwarz inequality.  We prove the statement for step (6) below.

We want to show
\[ \sum_{\sigma \sim \tau} |f(\sigma)^2 - f(\tau)^2| = \sum_{i=0}^{N-1}(f(\sigma_{i+1})^2 - f(\sigma_i)^2)|C_i|. \]
This can be seen by counting the number of times each $f(\sigma_i)^2$ appears in each sum.  In the left hand sum, each $f(\sigma_i)^2$ appears a number of times equal to 
\[ \Delta_i := \left| \set{ \set{\sigma_j, \sigma_i} : j < i \text{ \ and \ } \sigma_j \sim \sigma_i}| - |\set{ \set{\sigma_i, \sigma_k} : i < k \text{ \ and \ } \sigma_i \sim \sigma_k} \right|. \]
On the other hand, each $f(\sigma_i)^2$ appears $|C_{i-1}| - |C_i|$ times in the right hand sum.  To see that these are the same, note that for each pair $\set{\sigma_j, \sigma_k}$ in $C_{i-1}$, either $k=i$ or else $\set{\sigma_j, \sigma_k}$ is in $C_i$ as well, meaning it is canceled in the difference.  Similarly, for each pair $\set{\sigma_j, \sigma_k}$ in $C_i$, either $j=i$ or else $\set{\sigma_j, \sigma_k}$ is in $C_{i-1}$ as well, again meaning it is canceled.  Thus
\[ |C_{i-1}| - |C_i|  = \Delta_i. \]
This completes the proof.
\end{proof}
%\begin{lemma}
%\[ \sum_{i=0}^{N-1}(f(s_{i+1})^2 - f(s_i)^2)(N-i) = \sum_{s \in S_m} f(s)^2 \]
%\end{lemma}
%\begin{proof}
%\begin{align*}
% \sum_{i=0}^{N-1}(f(s_{i+1})^2 - f(s_i)^2)(N-i) &= \sum_{i=1}^N f(s_i)^2[(N-i+1) - (N-i)] \\
%                                                &= \sum_{s \in S_m} f(s)^2 
%\end{align*}

%\end{proof}

 We now state the negative result -- the analogous Cheeger-type inequality for the cochain complex does not hold. 

\begin{proposition}\label{cochain}
  For every $m>1$, there exist families of simplicial $m$-balls $X_k$ and $Y_k$ such that
  \begin{enumerate}
  \item[(1)] for $X_k$,  $\lambda^{m-1}(X_k) \geq \frac{(m-1)^2}{2(m+1)}$ for all $k$ but $h^{m-1}(X_k) \to 0$ as $k \to \infty$.
 \item[(2)] for $Y_k$, $\lambda^{m-1}(Y_k) \leq \frac{1}{m^{k-1}}$ for $k > 1$ but $h^{m-1}(Y_k) \geq \frac{1}{k}$ for all $k$.
  \end{enumerate}
  \end{proposition}

As mentioned in the introduction, it has already been shown in \cite{gundert2012laplacians} that there exist infinite families of simplicial complexes for which $h^{m-1} = 0$ but $\lambda^{m-1}$ is bounded away from 0.  Such a construction relies on the presence of torsion in the integral homology groups.  Indeed, any simplicial complex with torsion can be used to show that the inequality $(h^k)^p \geq C\lambda^k$ need not hold in general for any $p$,$C>0$, and $k>0$.  A good example is $\mathbb{RP}^2$ which has $H^1(\Z_2) \neq 0$ and $H^2(\Z_2) \neq 0$ but $H^1(\R) = 0$ and $H^2(\R) = 0$.  By contrast, the example presented here is a family of orientable simplicial complexes, proving that the failure of the Cheeger inequality to hold is not simply the result of torsion.  

The fact that both families $X_k$ and $Y_k$ are simplicial $m$-balls
 helps show the degree
to which the Cheeger inequality fails to hold even for `nice'
simplicial complexes.  

%% Second, it is in keeping with the original spirit of Cheeger
%% numbers as a concept that lay in the intersection between geometry
%% and combinatorics.  If similar questions about higher-dimensional
%% Cheeger numbers are raised in the manifold setting, the hope is
%% that these examples will give some guidance in testing continuous
%% Cheeger inequalities as well.  The original Cheeger inequality did
%% not have any combinatorics in mind.  You must be very careful here
%% talking about higher dimensions, much has been attempted but it is
%% very hard.

The proof of Proposition \ref{cochain} puts together much of what appears earlier in this paper.  To show that $X_k$ is a simplicial $m$-ball we will need to prove that it is constructible and non-branching.  The $Y_k$ will be defined by subdividing $\Sigma^m$, implying that it too is a simplicial $m$-ball.  To compute the values of $h^{m-1}$ for $X_k$ and $Y_k$ we make use of Lemmas \ref{lem: depth-attaining cochains} and \ref{lem: X_k}.  Computing $h_m$ will involve simple counting.  By Theorem \ref{Buser_Cheeger} and the fact that $\lambda_m = \lambda^{m-1}$, we can use our estimate of $h_m$ to estimate $\lambda^{m-1}$, finishing the proof.  

Now to begin the proof.  We define the family $X_k$ recursively.  To begin with, we let $X_1$ be $\Sigma^m$, the simplicial complex induced by a single $m$-simplex.  Note that $h_m(X_1) = m+1$ and $h^{m-1}(X_1) = 1$.  Then, given $X_k$, we define $X_{k+1}$ by gluing $m$-simplexes on to $X_k$ as follows: for each boundary face $s = \set{v_0,\ldots,v_{m-1}}$ in $X_k$ we create a new vertex $v$ and a new $m$-simplex $\sigma = \set{v_0,\ldots,v_{m-1}, v}$ which includes $v$ and $s$.  A picture of the first few iterations of $X_k$ for the case $m=2$ can be seen in Figure \ref{X_k}.

\begin{figure}
\begin{tikzpicture}
 \node[anchor=south west,inner sep=0] (image) at (0,0) {\includegraphics[width=0.9\textwidth]{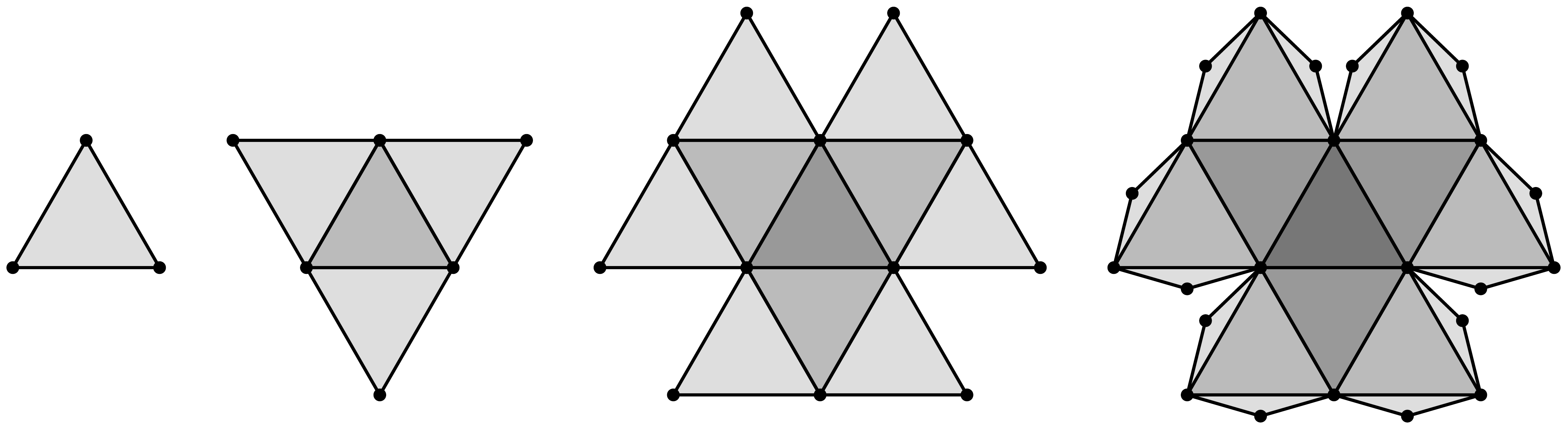}};
 \begin{scope}[x={(image.south east)},y={(image.north west)}]
 %\draw[help lines,xstep=.1,ystep=.1] (0,0) grid (1,1);
 \node [anchor=north] at (0.06,0) {$X_1$}; 
 \node [anchor=north] at (0.245,0) {$X_2$}; 
 \node [anchor=north] at (0.525,0) {$X_3$}; 
 \node [anchor=north] at (0.855,0) {$X_4$}; 
 \end{scope}
\end{tikzpicture}
\caption{The first few iterations of $X_k$ in dimension 2.  The 2-simplexes have been shaded according to their depth.}
\label{X_k}
\end{figure}
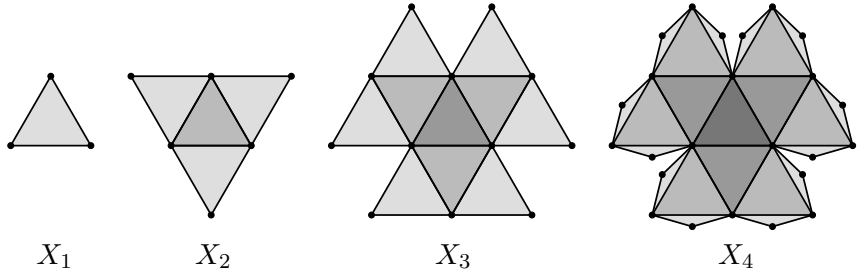

Clearly, $X_1$ is a simplicial $m$-ball.  The following two lemmas prove that indeed every $X_k$ is a simplicial $m$-ball.
\begin{lemma}
$X_k$ is constructible for all $k$.
\end{lemma}

\begin{proof}
The proof is by induction.  We know $X_1$ is constructible.  Assuming that $X_k$ is constructible, we must prove that $X_{k+1}$ is constructible.  This reduces to proving that gluing a single $m$-simplex to $X_k$ along a boundary face preserves constructibility.  Let $X'_k$ be the result of taking a boundary face $s = \set{v_0,\ldots,v_{m-1}}$ in $X_k$ and adding a new vertex $v$ and a new $m$-simplex $\sigma = \set{v_0,\ldots,v_{m-1}, v}$ which includes $v$ and $s$.  Then $X'_k$ can be decomposed as the union of $X_k$ and the simplicial subcomplex $T = \Sigma^m$ consisting of $\sigma$ and its subsets, both of which are constructible $m$-complexes.  Furthermore, the intersection of $X_k$ and $T$ is $\Sigma^{m-1}$, which is constructible.  Therefore, $X'_k$ is constructible by definition.
\end{proof}

\begin{lemma}
$X_k$ is non-branching for all $k$.
\end{lemma}

\begin{proof}
The proof is again by induction.  We know that $X_1$ is non-branching.  Assume this is true for $X_k$ as well.  By construction, $s \in S_{m-1}(X_k)$ has another coface in $S_m(X_{k+1})$ if and only if $s$ has only one coface in $S_m(X_k)$.  The new $(m-1)$-simplexes are the boundary faces of $X_{k+1}$ and thus have exactly one coface.  Thus, the total number of cofaces of every $(m-1)$-simplex in $X_{k+1}$ is either one or two.  
\end{proof}

As mentioned in the introduction, constructible non-branching simplicial $m$-complexes are simplicial $m$-balls.  Thus, every $X_k$ is a simplicial $m$-ball.

To prove part (1) of Proposition \ref{cochain}, we need to keep track of how the Cheeger numbers $h^{m-1}(X_k)$ and $h_m(X_k)$ change with $k$.  This is accomplished in the following two lemmas.

\begin{lemma}
$h^{m-1}(X_k) = \frac{1}{k}$ for all $k$.
\end{lemma}
\begin{proof}
By Lemma \ref{lem: X_k}, $h^{m-1}(X_k) = \frac{1}{\rad(X_k)}$.  For $k=1$, $\rad(X_1) = 1$.  Now suppose that $\rad(X_k) = k$.  We will prove that in passing from $X_k$ to $X_{k+1}$, all $m$-simplexes originally in $X_k$ have their depth increased by exactly 1 (we already know the new $m$-simplexes in $X_{k+1}$ have depth 1).  

If $\tau \in S_m(X_k)$ has depth $d$ and $\phi$ is a depth-attaining cochain for $\tau$ in $X_k$, then $\phi$ is a sum of a sequence $\set{s_i}_{i=1}^d$ of $(m-1)$-simplexes satisfying the conditions in Lemma \ref{lem: depth-attaining cochains}.  All of those conditions are preserved in going from $X_k$ to $X_{k+1}$, except that $s_d$ is no longer a boundary face.  Instead, if $s_d = \set{v_0, \ldots, v_{m-1}}$ then a new vertex $v$ and a new $m$-simplex $\sigma = \set{v_0, \ldots, v_{m-1}, v}$ are created which prevent $s_d$ from being a boundary face and add $\sigma$ to the coboundary of $\phi$.  However, if we add any of the other faces of $\sigma$ to $\phi$ (which are all boundary faces), we obtain a new cochain $\phi'$ with $\delta \phi' = \tau$ and $\norm{\phi'} = d+1$.  Thus, the depth of $\tau$ in $X_{k+1}$ is at most $d+1$.

Conversely, if $\tau$ has depth $d'$ in $X_{k+1}$ and $\psi = \sum_{i=1}^{d'} t_i$ is a depth-attaining cochain for $\tau$ with $\set{t_i}_{i=1}^{d'}$ satisfying the conditions in Lemma \ref{lem: depth-attaining cochains}, then $\psi' = \sum_{i=1}^{d'-1} t_i$ is a cochain in $X_k$ with $\delta \psi' = \tau$, so that the depth of $\sigma$ is at most $d'-1$.  Thus, if $\tau$ has depth $d$ in $X_k$ then its depth in $X_{k+1}$ must be at least $d+1$.  Combined with the above result we conclude that all $m$-simplexes originally in $X_k$ have their depth increased by exactly 1 in $X_{k+1}$.
\end{proof}
\begin{lemma}
$h_m(X_k) \geq m-1$ for all $k$.
\end{lemma}
\begin{proof}
We know that $h_m(X_1) = m+1 \geq m-1$.  Now suppose $h_m(X_k) \geq m-1$.  Any chain $\phi \in C_m(\Z_2;X_{k+1})$ attaining $h_m$ can be decomposed into a chain $\psi \in C_m(\Z_2;X_{k})$ plus a chain $\psi'$ which is a sum of depth 1 simplexes in $X_{k+1}$.  Then we can write $\norm{\partial \phi} = \norm{\partial \psi} + \norm{\partial \psi'} - 2x$ where $x$ is the number of $(m-1)$-simplexes shared by $\partial \psi$ and $\partial \psi'$.  Since $m$ of the $m+1$ faces of any $m$-simplex in $\psi'$ are boundary faces, $x \leq \norm{\psi'}$.  Also, it is clear that $\norm{\partial \psi'} = (m+1)\norm{\psi'}$.  Thus,
\begin{align*}
\frac{\norm{\partial \phi}}{\norm{\phi}} &= \frac{\norm{\partial \psi} + (m+1)\norm{\psi'} - 2x}{\norm{\psi} + \norm{\psi'}} \\
                                       &\geq \frac{\norm{\partial \psi} + (m-1)\norm{\psi'}}{\norm{\psi} + \norm{\psi'}} \\
                                       &\geq \min\left\{ \frac{\norm{\partial \psi}}{\norm{\psi}}, m-1 \right\} \\
                                       &\geq m-1
\end{align*}
(In fact, with some effort it can be seen that $h_m = \frac{(m+1)(m-1)}{(m+1) - 2m^{-k+1}}$.)  
\end{proof}
By Theorem \ref{Buser_Cheeger}, $\lambda^{m-1}(X_k) = \lambda_m(X_k) \geq \frac{(m-1)^2}{2(m+1)}$.  This completes the proof of part (1) of Proposition \ref{cochain}.

In order to define the family $Y_k$ we need to make use of the notion of stellar subdivision, which can be traced back to at least \cite{alexander1930combinatorial}.

\begin{definition}[Stellar Subdivision]\label{def: Y_k}
Let $Y$ be a simplicial $m$-complex and let $\sigma = \set{v_0, \ldots, v_m} \in S_m(Y)$.  The stellar subdivision of $Y$ along $\sigma$, denoted by $\sd_\sigma Y$, is the simplicial $m$-complex obtained from $Y$ by creating a new vertex $w$ and replacing $\sigma$ with the $m$-simplexes
\[ \tau_i = \set{v_0, \ldots, v_{i-1}, w, v_{i+1}, \ldots, v_m} \]
where $i = 0, \ldots, m$.  For notational purposes, we denote the $j$-th face of $\tau_i$ by $t_{i,j} := \tau_i \setminus \set{v_j}$ for $i \neq j$, and $t_{i,i} := \tau_i \setminus \set{w}$.  If $\sigma_1, \ldots, \sigma_n \in S_m(Y)$, then we define the stellar subdivision of $Y$ along the $\sigma_i$ to be 
\[ \sd_{\sigma_1, \ldots, \sigma_n} Y := \sd_{\sigma_1}\sd_{\sigma_2}\cdots\sd_{\sigma_n} Y \]
\end{definition}

We now define the $Y_k$ recursively.  Let $\Sigma^m$ be the simplicial complex induced by a single $m$-simplex $\sigma$ and let $Y_1 := \sd_\sigma \Sigma^m$.  Label the $m$-simplexes of $Y_1$ as $\sigma_0, \ldots, \sigma_m$ and call their common vertex (the one created by stellar subdivision) the central vertex $v$.  Now, given a $Y_k$ containing the central vertex $v$, we call all $m$-simplexes containing $v$ the inner $m$-simplexes of $Y_k$ and label them as $\sigma_0, \ldots, \sigma_n$.  All non-inner $m$-simplexes will be referred to as outer $m$-simplexes.  We then define $Y_{k+1} := \sd_{\sigma_0, \ldots, \sigma_n} Y_k$.  Note that $v$ and all outer $m$-simplexes (and the simplexes they contain) are preserved unchanged in going from $Y_k$ to $Y_{k+1}$ while all of the inner $m$-simplexes are subdivided.  Furthermore, it is clear that all the $Y_k$ are subdivisions of $\Sigma^m$ and are thus simplicial $m$-balls.  A picture of the first few iterations of $Y_k$ for $m=2$ can be seen in Figure \ref{Y_k}.

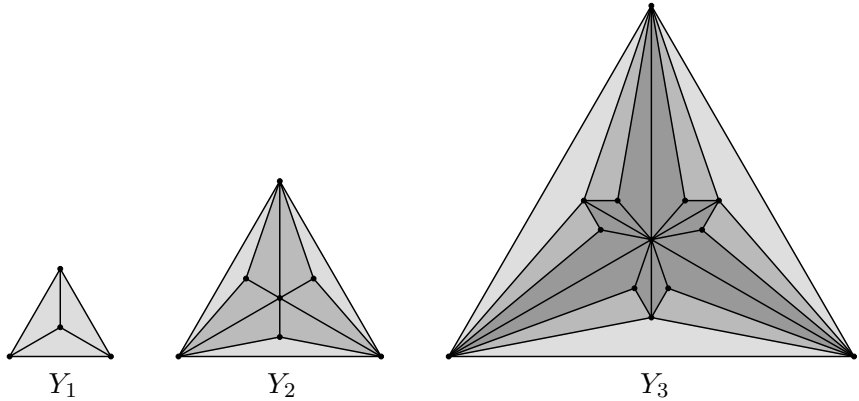
\begin{figure}
\begin{tikzpicture}
 \node[anchor=south west,inner sep=0] (image) at (0,0) {\includegraphics[width=0.9\textwidth]{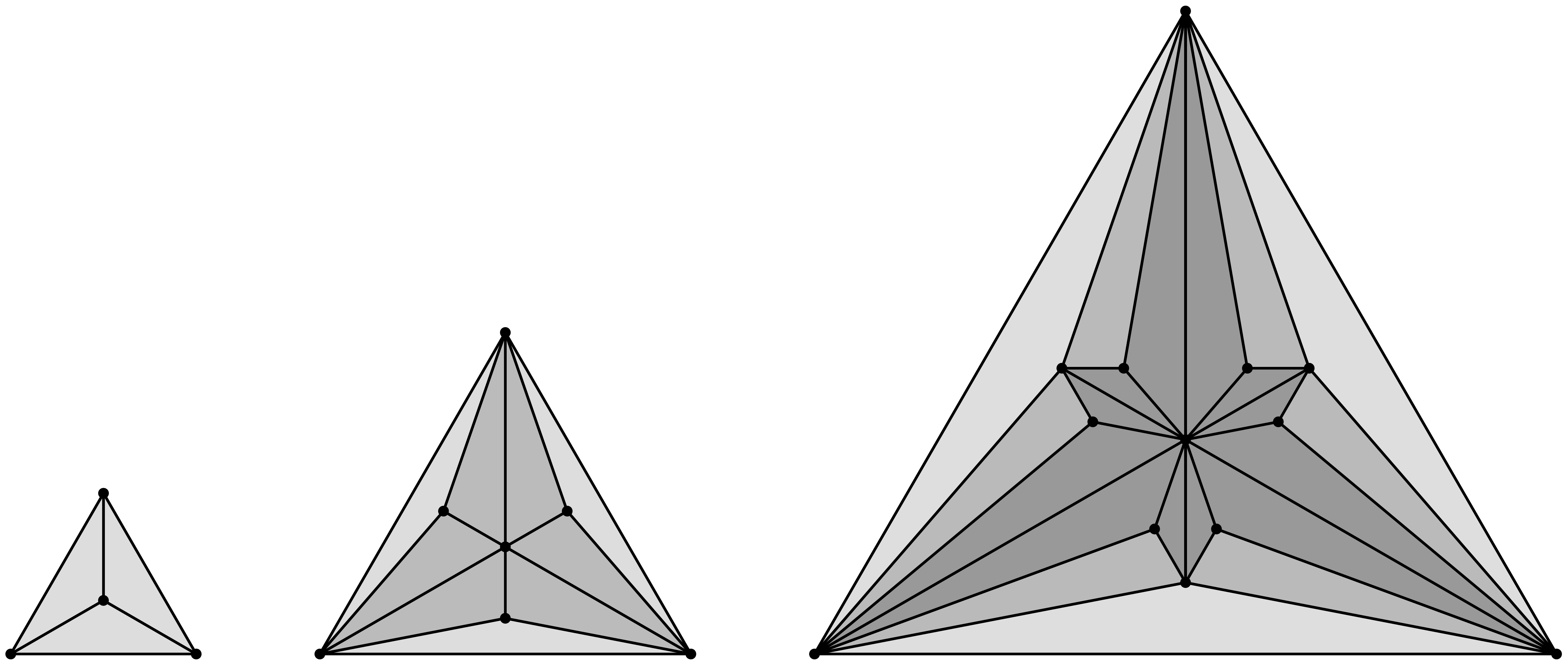}};
 \begin{scope}[x={(image.south east)},y={(image.north west)}]
 %\draw[help lines,xstep=.1,ystep=.1] (0,0) grid (1,1);
 \node [anchor=north] at (0.07,0) {$Y_1$}; 
 \node [anchor=north] at (0.325,0) {$Y_2$}; 
 \node [anchor=north] at (0.76,0) {$Y_3$}; 
 \end{scope}
\end{tikzpicture}
\caption{The first few iterations of $Y_k$ in dimension 2.  The $2$-simplexes have been shaded according to their depth.}
\label{Y_k}
\end{figure}

To prove part (2) of Proposition \ref{cochain}, we need to keep track of how the Cheeger numbers $h^{m-1}(Y_k)$ and $h_m(Y_k)$ change with $k$.  This is accomplished in the following two lemmas.

\begin{lemma}
$h^{m-1}(Y_k) \geq \frac{1}{k}$ for all $k$.
\end{lemma}
\begin{proof}
By Lemma \ref{lem: X_k}, we can prove this by keeping track of the depths of all the $m$-simplexes of $Y_k$.  For $Y_1$, all the $m$-simplexes $\sigma_i$ contain a boundary face (using the notation of Definition \ref{def: Y_k} with $\sigma_i = \tau_i$, the boundary face of $\sigma_i$ is $t_{i,i}$).  Thus, every $\sigma_i$ has depth 1 and by Lemma \ref{lem: X_k}, $h^{m-1}(Y_1) = 1$.  Note that the cochain $\phi$ which is depth-attaining for some $\sigma_i$ does not include any $(m-1)$-simplex which contains $v$.  

Now suppose for induction that every outer $m$-simplex $\sigma$ of $Y_k$ has depth $\leq k$ and a depth-attaining cochain $\phi \in C^{m-1}(\Z_2)$ such that $\phi$ does not contain any face of any inner $m$-simplex.  Then in $Y_{k+1}$, $\phi$ remains unaltered, proving that $\sigma$ still has depth $\leq k$ in $Y_{k+1}$.  

Similarly, suppose that every inner $m$-simplex $\sigma$ of $Y_k$ has depth $\leq k$ via a depth-attaining cochain $\phi$ which does not contain any $(m-1)$-simplex containing $v$.  Then in $Y_{k+1}$, $\sigma$ is removed and replaced by new $m$-simplices.  Using the notation of Definition \ref{def: Y_k}, in $Y_{k+1}$ the coboundary of $\phi$ becomes $\delta \phi = \tau_{m+1}$, so that the depth of $\tau_{m+1}$ is at most $k$.  Furthermore, by adding any face $t_{(m+1),j}$ to $\phi$ ($j \neq m+1$) we obtain a cochain $\phi'$ with $\delta \phi' = \tau_j$, proving that the depth of $\tau_j$ is at most $k+1$.  Since $\phi'$ still does not contain any $(m-1)$-simplex which contains $v$, we are back where we started.  The statement now follows by induction.
\end{proof}

\begin{lemma}
$h_m(Y_k) \leq \frac{1}{m^{k-1}}$ for all $k>1$.
\end{lemma}
\begin{proof}
To prove this, we merely count the number of $m$-simplexes in $Y_k$.  Note that in going from $Y_k$ to $Y_{k+1}$ we replace $(m+1)m^{k-1}$ inner $m$-simplexes with $(m+1)m^k$ inner $m$-simplexes.  Thus, $Y_{k+1}$ has 
\[(m+1)m^k - (m+1)m^{k-1} = (m+1)(m-1)m^{k-1}\]
more $m$-simplexes than $Y_k$.  Since $|S_m(Y_1)| = m+1$, this means that $|S_m(Y_k)|$ is equal to
\begin{multline*}
(m+1) + (m+1)(m-1) + (m+1)(m-1)m + \ldots \\+ (m+1)(m-1)m^{k-2} = (m+1)m^{k-1}.
\end{multline*}
Since $Y_k$ has $m+1$ boundary faces, the chain $\phi$ containing all $m$-simplexes of $Y_k$ gives the upper bound on $h_m(Y_k)$:
\[ h_m(Y_k) \leq \frac{\norm{\partial \phi}}{\norm{\phi}} = \frac{m+1}{(m+1)m^{k-1}} = \frac{1}{m^{k-1}}. \]
\end{proof}
By Theorem \ref{Buser_Cheeger}, $\lambda^{m-1}(Y_k) = \lambda_m(Y_k) \leq \frac{1}{m^{k-1}}$.  This completes the proof of Proposition \ref{cochain}.

\section{Discussion and Open Problems}

The Cheeger inequality has been relevant to a variety of algorithmic 
and analysis problems in computer science and mathematics 
including spectral clustering 
\cite{kannan2004clusterings,Shi_2001_3807}, manifold learning
\cite{belkin2008towards}, and the analysis of random walks
\cite{lawlersocal88}. 

There has been interest in extending ideas from graphs to abstract
simplicial complexes including spanning trees on simplicial complexes
\cite{DuKlMa}, properties of expanders on simplicial complexes 
\cite{meshulam2009homological,dotterrer2010coboundary,gundert2012laplacians},
and higher-dimensional constructions of conditional independence \cite{Luna09}.
A motivation for our work was to begin to develop intuition for the
mathematical principles behind a higher-dimensional notion of spectral 
clustering. This objective is far from being realized.

A result of the universal coefficient theorem in algebraic topology
is that torsion will be an obstacle in relating higher-dimensional 
Cheeger numbers with eigenvalues. The Cheeger inequality for graphs
holds without any assumptions since zeroth homology is never affected
by torsion.  For higher dimensions either the inequality does not hold 
or we require assumptions that remove torsion. The negative results 
for the Cheeger inequality in \cite{gundert2012laplacians} are for
simplicial complexes with torsion. Torsion is also known to affect
algorithmic complexity. For example, the problem of finding minimal 
weight cycles given a simplicial complex with weights is
NP-hard if there is torsion and is otherwise a linear program 
\cite{DeyHiKris}. In Appendix A we use the real projective plane to illustrate some of the issues 
with torsion and why they do not appear in the graph setting. 

A local Cheeger number and algebraic connectivity for graphs with
Dirichlet like boundary conditions was defined in 
\cite{chung2007random} and a Cheeger inequality was proved.
There is a close relation between Theorem 1 of \cite{chung2007random}
and Theorem \ref{Buser_Cheeger} in our paper. If Theorem 1
is adapted to an unnormalized setting (see
Appendix A) then for non-branching orientable simplicial $m$-complexes 
Theorem \ref{Buser_Cheeger} reduces to Theorem 1. However, Theorem
\ref{Buser_Cheeger} covers the more general cases of non-orientable 
and branching simplicial $m$-complexes.

We close with a few open problems of possible interest.
\begin{enumerate}
\item Intermediate values of $k$ -- Given a simplicial
$m$-complex, what can we say about the relationship between $h^k$
and $\lambda^k$ or $h_k$ and $\lambda_k$ for $1 < k < m-1$?
Torsion again will need to be addressed but are there some
conditions under which some Cheeger-type 
inequalities may hold? 
\item High-order eigenvalues -- In \cite{lee2011multi} the
authors introduce higher-order (as opposed to higher-dimensional)
Cheeger numbers on the graph which correspond to higher-order
eigenvalues of the graph Laplacian and prove a general Cheeger
inequality for them.  A natural question is how our results would
extend to higher-orders.
Indeed, by analogy with the Rayleigh quotient characterization of higher order eigenvalues, it would seem reasonable to define the $k^\text{th}$ dimensional, $j^\text{th}$ order coboundary Cheeger numbers to be
\[ h^{k,j} := \min_{\substack{\phi \in C^k(\Z_2) \\ \phi \notin S_j}} \frac{\norm{\delta \phi}}{\min_{\psi \in S_j} \norm{\phi + \psi}} \]
where 
\[ S_j = \spn(\im \delta \cup \set{\phi_1, \ldots, \phi_{j-1}}) \]
is the subspace of $C^k(\Z_2)$ spanned by $\im \delta$ and cochains $\phi_1, \ldots, \phi_{j-1}$ which attain $h^{k,1}, \ldots, h^{k,j-1}$, respectively.  The higher order boundary Cheeger numbers $h_{k,j}$ could be defined similarly.  One would need to prove that this definition makes sense and then ask whether they satisfy any inequalities with the corresponding eigenvalues.
\item Cheeger inequalities on manifolds --
Ultimately, the study of higher-dimensional Cheeger numbers on simplicial complexes should (morally speaking) be translated back to the manifold setting if possible.  A tentative definition for the $k$-dimensional coboundary  Cheeger number of a manifold $M$ might be
\[ h^k = \inf_S \frac{\Vol_{m-k-1}(\partial S \setminus \partial M)}{\inf_{\partial T = \partial S} \Vol_{m-k}(T)} \]
where $\Vol_k$ denotes $k$-dimensional volume and the infimum is taken over all $k$-codimensional submanifolds $S$ of $M$.
Similarly, the $k$-th boundary Cheeger number of $M$ might be
\[ h_k = \inf_S \frac{\Vol_{k-1}(\partial S)}{\inf_{\partial T = \partial S} \Vol_k(T)} \]
where again $\Vol_k$ denotes $k$-dimensional volume and the infimum is taken over all $k$-dimensional submanifolds $S$ of $M$.
\end{enumerate}

\section*{Acknowledgments}
SM would like to acknowledge Shmuel Weinberger and Matt Kahle for discussions and
insight. SM is pleased to acknowledge support from grants NIH (Systems Biology): 5P50-GM081883, AFOSR:
FA9550-10-1-0436, NSF CCF-1049290, and NSF-DMS-1209155.  JS would like to acknowledge Matt Kahle, Yuan Yao, Anna Gundert, Yuriy Mileyko, and Mikhail Belkin for discussions and insight.  JS is pleased to acknowledge support from graph NSF CCF-1209155 and a Duke Endowment Fellowship.

\appendix
\section{Relation to Graphs with Dirichlet Boundaries and the Real Projective Plane}

In \cite{chung2007random}, Fan Chung defines a normalized local Dirichlet Cheeger number and normalized local Dirichlet eigenvalue and proves an inequality between them.  If one translates Fan Chung's result to the unnormalized case for graphs with vertex degree upper bounded by $m+1$, it closely resembles 
Theorem \ref{Buser_Cheeger}. 

Translating Theorem 1 of \cite{chung2007random} into the unnormalized setting, it reads as follows.  Given a graph $G$ we can prescribe a certain set of vertices to be the boundary vertices of the graph.  Let $S$ be the prescribed boundary vertex set, and let 
\[ h_S := h_S(G) = \min \frac{\norm{\delta \phi}}{\norm{\phi}} \]
where the minimum is taken over all nonzero $\phi \in C^0(\Z_2)$ such that $\phi$ does not include any boundary vertex.  Similarly, let \[ \lambda_S = \min \frac{\norm{\delta f}_2^2}{\norm{f}_2^2} \]
where the minimum is taken over all nonzero $f \in C^0(\R)$ such that $f(s) = 0$ for all $s \in S$.  We can also characterize $\lambda_S$ as the smallest eigenvalue of $L_0^S$, the submatrix of $L_0$ consisting of the rows and columns of $L_0$ not indexed by vertices in $S$.  In this case, $L_0^S$ is a map on $C^0_S(\R)$, the subspace of $C^0(\R)$ spanned by the vertices not in $S$. Then if every vertex has degree upper bounded by $m+1$
\[ h_S \geq \lambda_S \geq \frac{h_S^2}{2(m+1)}.\]

To relate the above inequality to the simplicial complex setting, we note that for every non-branching simplicial $m$-complex $X$, one can construct a graph $G$ (similar to the dual graph defined in \cite{fomin2008cluster}) as follows.  Begin by constructing the simplicial complex $X'$ as in the proof of Theorem \ref{Buser_Cheeger} and let $S$ be the set of border facets of $X'$.  Create a vertex in $G$ for every $m$-simplex in $X'$.  We will use $S$ to denote both the border facets of $X'$ and the set of vertices in $G$ which correspond the border facets.  Connect two vertices with an edge whenever the corresponding $m$-simplexes are lower adjacent in $X'$.  Since $X'$ is non-branching, the vertices of $G$ have degree upper bounded by $m+1$.  Identifying $C^0_S(G; \R)$ with $C_m(X; \R)$, we can ask if $L_m : C_m(X; \R) \to C_m(X; \R)$ and $L_0^S : C^0_S(G; \R) \to C^0_S(G; \R)$ are the same map.  They are the same if and only if $X$ is orientable (this is easy to see from the characterization of the Laplacian in \cite{muhammad2006control}).  In addition, $h_m(X)$ and $h_S$ are equal regardless of orientability.  Thus, for non-branching orientable simplicial $m$-complexes, Theorem \ref{Buser_Cheeger} reduces to the result proved by Fan Chung, and the proofs are identical.  The difference is that Theorem \ref{Buser_Cheeger} covers the more general cases of non-orientable and branching simplicial $m$-complexes, for which parts of the inequality may still hold.

The real projective plane provides a simple example of how orientation plays a role in our analysis of the Cheeger inequality and why it doesn't play a role in \cite{chung2007random}.  In Figure \ref{RP2_Comp_Graph}, the first image shows the fundamental polygon that defines $\R P^2$, the second image shows a triangulation $X$ of $\R P^2$, and the third image is the dual graph $G$ of the triangulation (in the second and third image, edges with similar color are identified).  In this simple example, there is no boundary ($S = \emptyset$).  In the triangulation, if one considers the 2-chain $\phi \in C_2(\Z_2)$ which contains every 2-simplex, then $\partial \phi = 0$ and thus $h_2(X) = 0$.  However, if one considers the 2-chain $f \in C_2(X; \R)$ that assigns a 1 to every 2-simplex with the orientation shown in the figure, the boundary of $f$ is a 1-chain which assigns a 2 to every colored edge with the orientation shown.  In particular, $\partial f \neq 0$ and in fact $\lambda_2 \neq 0$ as a result of the nonorientability of $\R P^2$.  However, the dual graph cannot see this nonorientability, as the 0-chain $\tilde{f} \in C_S^0(G; \R)$ corresponding to $f$ has empty coboundary, meaning $\lambda_S = 0$.  Thus, in this case the map $L_2$ is not the same as the map $L_0^S$, and Theorem 1 of \cite{chung2007random} still holds while part 1 of Theorem \ref{Buser_Cheeger} fails.

\begin{figure}
\centering
\includegraphics[width=0.3\textwidth]{RP2_Fund_Poly.pdf} \quad
\includegraphics[width=0.3\textwidth]{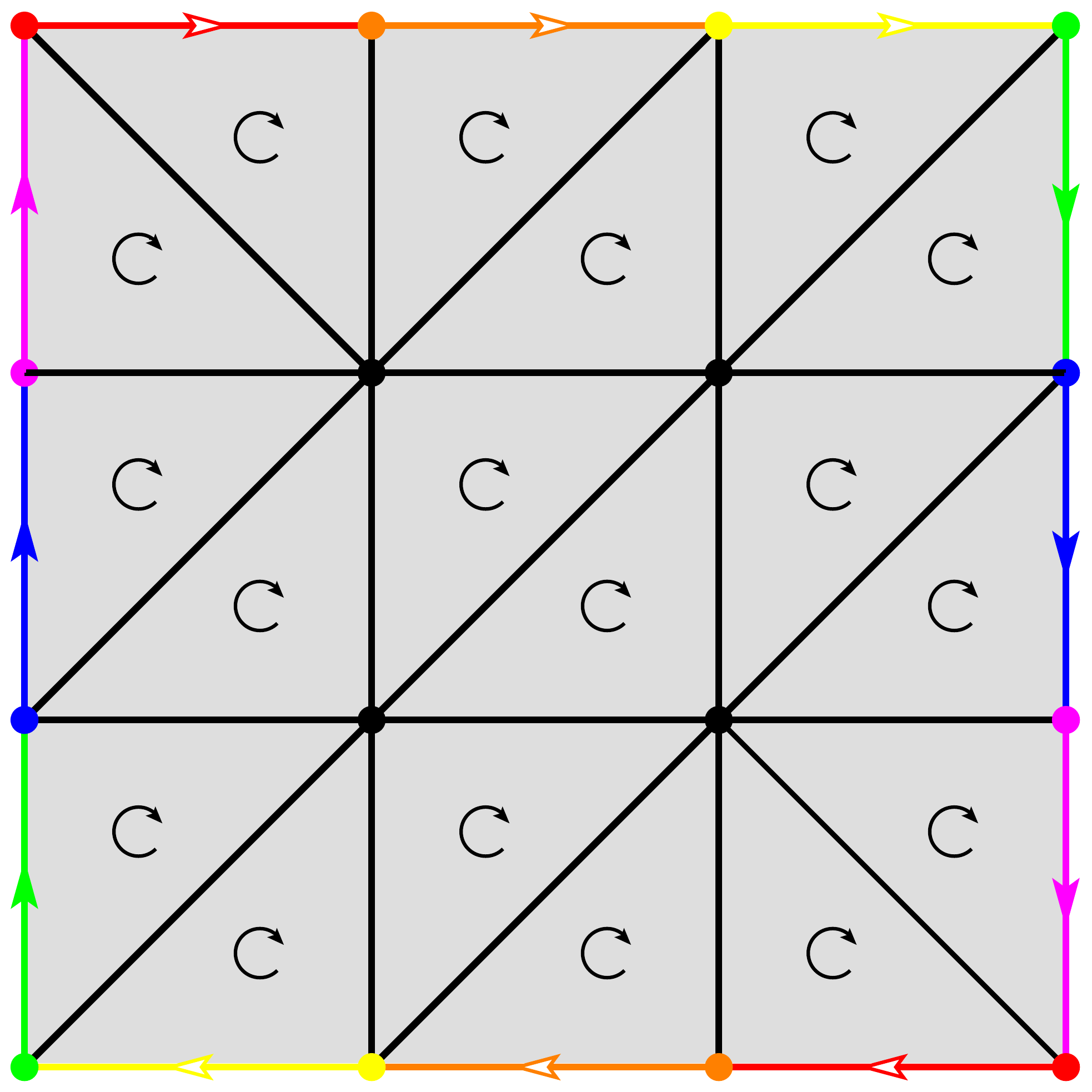} \quad
\includegraphics[width=0.3\textwidth]{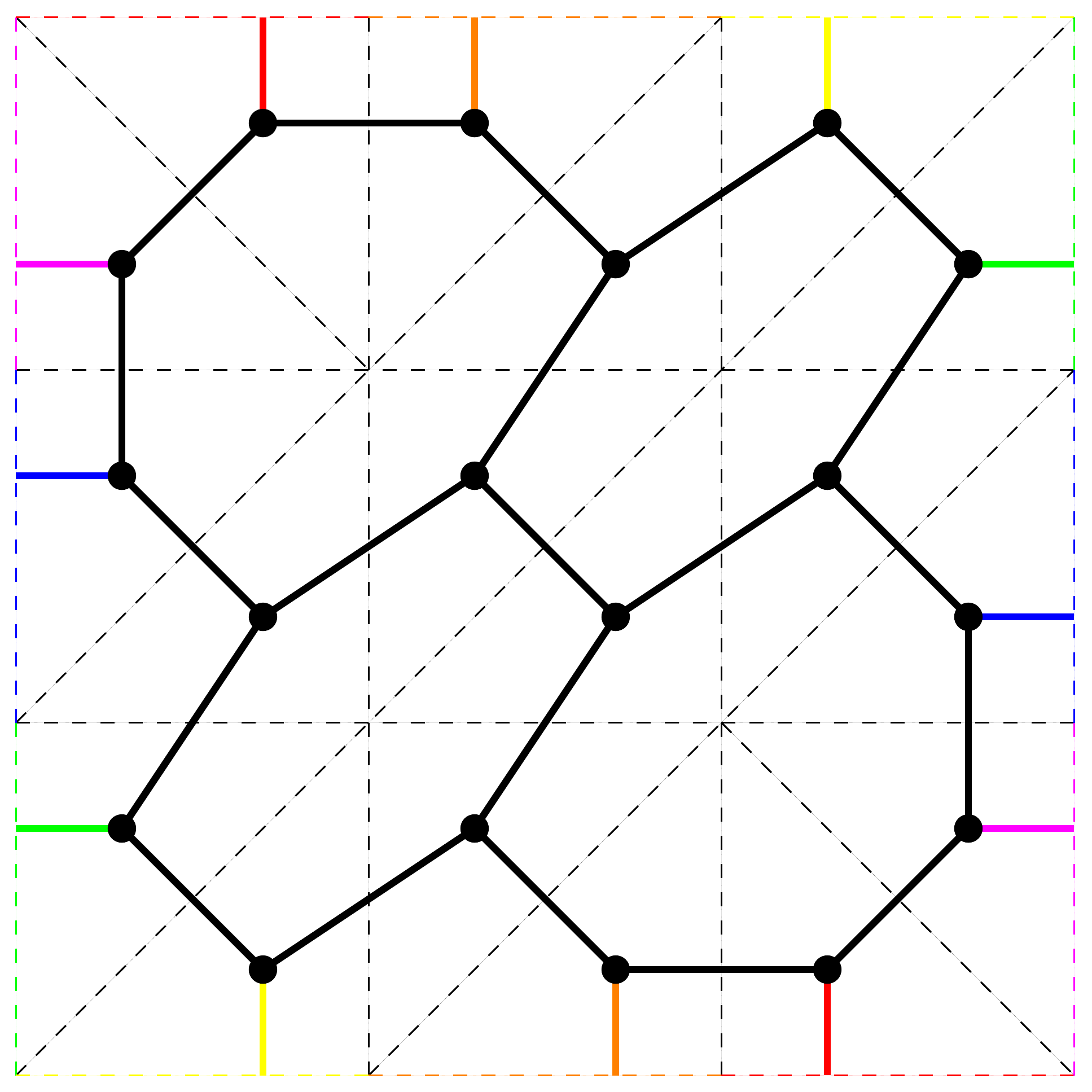}
\caption{The fundamental polygon of $\R P^2$, a triangulation, and the dual graph of the triangulation.}
\label{RP2_Comp_Graph}
\end{figure}

\bibliographystyle{plain}
\bibliography{bibliography}

\begin{thebibliography}{10}

\bibitem{alexander1930combinatorial}
J.W. Alexander.
\newblock The combinatorial theory of complexes.
\newblock {\em The Annals of Mathematics}, 31(2):292--320, 1930.

\bibitem{alon1986eigenvalues}
N.~Alon.
\newblock Eigenvalues and expanders.
\newblock {\em Combinatorica}, 6(2):83--96, 1986.

\bibitem{alon1985lambda}
N.~Alon and V.D. Milman.
\newblock $\lambda_1$, {I}soperimetric {I}nequalities for {G}raphs, and
  {S}uperconcentrators.
\newblock {\em Journal of Combinatorial Theory, Series B}, 38(1):73--88, 1985.

\bibitem{belkin2008towards}
M.~Belkin and P.~Niyogi.
\newblock Towards a theoretical foundation for laplacian-based manifold
  methods.
\newblock {\em Journal of Computer and System Sciences}, 74(8):1289--1308,
  2008.

\bibitem{bjorner1995topological}
A.~Bjorner.
\newblock Topological methods.
\newblock {\em Handbook of combinatorics}, 2:1819--1872, 1995.

\bibitem{buser36cheeger}
P.~Buser.
\newblock On {C}heeger's {I}nequality $\lambda_1 \geq h^2/4$.
\newblock In {\em Proc. Sympos. Pure Math}, volume~36, pages 29--77, 1980.

\bibitem{cheeger1970lower}
J.~Cheeger.
\newblock {A lower bound for the smallest eigenvalue of the Laplacian}.
\newblock {\em Problems in analysis}, pages 195--199, 1970.

\bibitem{chung2007random}
F.~Chung.
\newblock Random walks and local cuts in graphs.
\newblock {\em Linear Algebra and its applications}, 423(1):22--32, 2007.

\bibitem{chung1997spectral}
F.R.K. Chung.
\newblock {\em {Spectral graph theory}}.
\newblock Amer. Mathematical Society, 1997.

\bibitem{DeyHiKris}
T.K. Dey, A.N. Hirnai, and B.~Krishnamoorthy.
\newblock {Optimal Homologous Cycles, Total Unimodularity, and Linear
  Programming}.
\newblock {\em Arxiv preprint math/1001.0338}, 2011.

\bibitem{dotterrer2010coboundary}
D.~Dotterrer and M.~Kahle.
\newblock {Coboundary expanders}.
\newblock {\em Arxiv preprint arXiv:1012.5316}, 2010.

\bibitem{DuKlMa}
A.~Duval, C.J. Klivans, and J.L. Martin.
\newblock {Simplicial Spanning Trees and Generalized Matrix-Tree Theorems}.
\newblock {\em Trans. Amer. Math. Soc.}, 361, 2009.

\bibitem{eckmann1944harmonische}
B.~Eckmann.
\newblock Harmonische {F}unktionen und {R}andwertaufgaben in einem {K}omplex.
\newblock {\em Comm. Math. Helv.}, 17(1):240--255, 1944.

\bibitem{fiedler1973algebraic}
M.~Fiedler.
\newblock Algebraic connectivity of graphs.
\newblock {\em Czechoslovak Mathematical Journal}, 23(2):298--305, 1973.

\bibitem{fomin2008cluster}
S.~Fomin, M.~Shapiro, and D.~Thurston.
\newblock Cluster algebras and triangulated surfaces. part i: Cluster
  complexes.
\newblock {\em Acta Mathematica}, 201(1):83--146, 2008.

\bibitem{gundert2012laplacians}
A.~Gundert and U.~Wagner.
\newblock On {L}aplacians of random complexes.
\newblock In {\em Proceedings of the 2012 Symposuim on Computational Geometry},
  pages 151--160. ACM, 2012.

\bibitem{hoory2006expander}
S.~Hoory, N.~Linial, and A.~Wigderson.
\newblock {Expander graphs and their applications}.
\newblock {\em Bulletin of the American Mathematical Society}, 43(4):439, 2006.

\bibitem{kannan2004clusterings}
R.~Kannan, S.~Vempala, and A.~Vetta.
\newblock {On clusterings: Good, bad and spectral}.
\newblock {\em Journal of the ACM (JACM)}, 51(3):497--515, 2004.

\bibitem{lawlersocal88}
G.~Lawler and A.~Sokal.
\newblock Bounds on the $\ell^2$ spectrum of {M}arkov chains and {M}arkov
  processes: a generalization of {C}heeger's inequlity.
\newblock {\em Trans. Amer. Math. Soc.}, 309:557--580, 1988.

\bibitem{lee2011multi}
J.R. Lee, S.O. Gharan, and L.~Trevisan.
\newblock Multi-way spectral partitioning and higher-order cheeger
  inequalities.
\newblock {\em Arxiv preprint arXiv:1111.1055}, 2011.

\bibitem{linial2006homological}
N.~Linial and R.~Meshulam.
\newblock Homological connectivity of random 2-complexes.
\newblock {\em Combinatorica}, 26(4):475--487, 2006.

\bibitem{Luna09}
S.~Lunag\'{o}mez, S.~Mukherjee, and R.L. Wolpert.
\newblock {Geometric Representations of Hypergraphs for Prior Specification and
  Posterior Sampling}.
\newblock {\em Arxiv preprint math/0912.3648}, 2009.

\bibitem{Shi_2001_3807}
Marina Maila and Jianbo Shi.
\newblock A random walks view of spectral segmentation.
\newblock In {\em Proceedings of the Eighth International Workshop on
  Artificial Intelligence and Statistics}, 2001.

\bibitem{meshulam2009homological}
R.~Meshulam and N.~Wallach.
\newblock {Homological connectivity of random k-dimensional complexes}.
\newblock {\em Random Structures \& Algorithms}, 34(3):408--417, 2009.

\bibitem{mohar1989isoperimetric}
B.~Mohar.
\newblock Isoperimetric numbers of graphs.
\newblock {\em Journal of Combinatorial Theory, Series B}, 47(3):274--291,
  1989.

\bibitem{muhammad2006control}
A.~Muhammad and M.~Egerstedt.
\newblock {Control using higher order Laplacians in network topologies}.
\newblock In {\em Proceedings of the 17th International Symposium on
  Mathematical Theory of Networks and Systems, Kyoto, Japan}, pages 1024--1038,
  2006.

\bibitem{parzanchevski2012isoperimetric}
O.~Parzanchevski, R.~Rosenthal, and R.J. Tessler.
\newblock Isoperimetric inequalities in simplicial complexes.
\newblock {\em arXiv preprint arXiv:1207.0638}, 2012.

\end{thebibliography}

\end{document}